\RequirePackage[l2tabu, orthodox]{nag}

\documentclass[reqno,12pt]{amsart} 
\usepackage{amsmath,amssymb,amsthm,enumerate,mathrsfs,colonequals}
\usepackage{fullpage,url,xypic,tikz,xspace,setspace}

\newcommand{\scriptD}{\mathscr{D}}
\newcommand{\scriptF}{\mathscr{F}}

\usetikzlibrary{topaths}
\tikzset{>=stealth, text height=1.5ex, text depth=0.25ex}

\usepackage{color} 

\newcommand{\F}{{\mathbb{F}}}

\newcommand{\G}{{\mathbb{G}}}

\newcommand{\Q}{{\mathbb{Q}}}

\newcommand{\R}{{\mathbb{R}}}

\newcommand{\Z}{{\mathbb{Z}}}
\newcommand{\ZZ}{{\mathbb{Z}}}
\newcommand{\cO}{{\mathcal{O}}}
\newcommand{\cP}{{\mathcal{P}}}
\newcommand{\p}{{\mathfrak{p}}}

\newcommand{\calE}{\mathcal{E}}
\newcommand{\calF}{\mathcal{F}}
\newcommand{\calN}{\mathcal{N}}

\newcommand{\directsum}{\oplus}
\newcommand{\Directsum}{\bigoplus}
\newcommand{\isom}{\simeq}
\newcommand{\tensor}{\otimes}
\newcommand{\union}{\cup}
\newcommand{\To}{\longrightarrow}
\newcommand{\injects}{\hookrightarrow}
\newcommand{\surjects}{\twoheadrightarrow}
\newcommand{\Jacobi}[2]{\genfrac{(}{)}{}{1}{#1}{#2}}

\newcommand{\et}{\textup{\'et}}

\DeclareMathOperator{\Aut}{Aut}
\DeclareMathOperator{\Average}{Average}
\DeclareMathOperator{\Br}{Br}
\DeclareMathOperator{\Char}{char}

\DeclareMathOperator{\Cl}{Cl}
\DeclareMathOperator{\coker}{coker}
\DeclareMathOperator{\Disc}{Disc}
\DeclareMathOperator{\Frob}{Frob}
\DeclareMathOperator{\Gal}{Gal}

\DeclareMathOperator{\Pic}{Pic}
\DeclareMathOperator{\Prob}{Prob}
\DeclareMathOperator{\rank}{rank}
\DeclareMathOperator{\sgn}{sgn}
\DeclareMathOperator{\Spec}{Spec}

\newtheorem{theorem}{Theorem}[section]
\newtheorem{conjecture}[theorem]{Conjecture}
\newtheorem{lemma}[theorem]{Lemma}

\newtheorem{corollary}[theorem]{Corollary}
\newtheorem{proposition}[theorem]{Proposition}

\theoremstyle{definition}
\newtheorem{example}[theorem]{Example}

\theoremstyle{remark}
\newtheorem{remark}[theorem]{Remark}

\usepackage{microtype}

\usepackage[
	pdfauthor={Bruce W. Jordan, Zev Klagsbrun, Bjorn Poonen, Christopher Skinner, and Yevgeny Zaytman}, 
	pdftitle={Statistics of K-groups modulo p for the ring of integers of a varying quadratic number field},
]{hyperref}
\usepackage[alphabetic,lite,nobysame]{amsrefs} 

\begin{document}

\title{Statistics of $K$-groups modulo $p$ for the ring of integers of
  a varying quadratic number field}


\author{Bruce~W.~Jordan}
\address{Department of Mathematics, Baruch College, The City University
of New York, One Bernard Baruch Way, New York, NY 10010-5526, USA}
\email{bruce.jordan@baruch.cuny.edu}

\author{Zev~Klagsbrun}
\address{Center for Communications Research, 4320 Westerra Court, San Diego, CA 92121-1969}
\email{zevklagsbrun@gmail.com}

\author{Bjorn~Poonen}
\address{Department of Mathematics, Massachusetts Institute of Technology, Cambridge, MA 02139-4307, USA}
\email{poonen@math.mit.edu}
\urladdr{\url{http://math.mit.edu/~poonen/}}

\author{Christopher~Skinner}
\address{Department of Mathematics, Princeton University, Fine Hall, Washington Road, Princeton, NJ 08544-1000, USA}
\email{cmcls@princeton.edu}

\author{Yevgeny~Zaytman}
\address{Center for Communications Research, 805 Bunn Drive, Princeton, 
NJ 08540-1966, USA}
\email{ykzaytm@idaccr.org}

\date{August 17, 2025}

\thanks{B.P.\ was supported in part by National Science Foundation grants DMS-1601946 and DMS-2101040 and Simons Foundation grants \#402472 (to Bjorn Poonen) and \#550033.  C.S.\ was supported in part by National Science Foundation grant DMS-1301842
  and by the Simons Investigator grant \#376203 from the Simons Foundation.
This article has appeared in \emph{Tunisian J.\ Math.}\ \textbf{2} (2020), no.~2, 287--307.  The present version differs from the published version in that the proof of Lemma~\ref{L:Hochschild-Serre} has been corrected.}

\begin{abstract} 
For each odd prime $p$,
we conjecture the distribution of 
the $p$-torsion subgroup of $K_{2n}(\cO_F)$ 
as $F$ ranges over real quadratic fields,
or over imaginary quadratic fields.
We then prove that the average size of 
the $3$-torsion subgroup of $K_{2n}(\cO_F)$ 
is as predicted by this conjecture.
\end{abstract}

\maketitle

\section{Introduction}
 
The original Cohen--Lenstra heuristics \cite{Cohen-Lenstra1984} 
predicted, for each prime $p \ne 2$, 
the distribution of the $p$-primary part of $\Cl(F)$ 
as $F$ varied over quadratic fields of a given signature. 
More recent work developed heuristics
for other families of groups, including class groups of higher
degree number fields \cite{Cohen-Martinet1990}, 
Picard groups of function fields \cite{Friedman-Washington1989}, 
Tate--Shafarevich groups of elliptic curves \cites{Delaunay2001,Delaunay2007,Delaunay-Jouhet2014a}, 
Selmer groups of elliptic curves \cites{Poonen-Rains2012-selmer,Bhargava-Kane-Lenstra-Poonen-Rains2015}, 
and Galois groups of nonabelian unramified extensions 
of number fields and function fields \cites{Boston-Bush-Hajir2017,Boston-Wood2017}.

Let $\calF$ be a number field.
Let $\cO_\calF$ be the ring of integers of $\calF$.
For $m \ge 0$, the $K$-group $K_m(\cO_\calF)$ 
is a finitely generated abelian group.  
It is finite when $m$ is even and positive:
see \cite{Weibel2005}*{Theorem~7}.
Our goal is to study, for a fixed $m$ and odd prime $p$,
the $p$-torsion subgroup $K_m(\cO_\calF)_p$
as $\calF$ varies in a family of number fields, 
always ordered by absolute value of the discriminant.
As described in Section \ref{sec:odd K groups}, 
$K_m(\cO_\calF)_p$ is well understood for odd $m$.
Therefore we focus on the case $m=2n$.
Now suppose that $\calF$ is a quadratic field $F$.
The action of $\Gal(F/\Q)$ decomposes $K_{2n}(\cO_F)_p$ 
into $+$ and $-$ parts,
and we will see that the $+$ part is $K_{2n}(\Z)_p$, independent of $F$.
Therefore we focus on the variation of the $-$ part.

A Cohen--Lenstra style heuristic will lead us to the following conjecture,
involving the constants
\[
\alpha_{p,u,r} \colonequals
\frac{\prod_{i = r+1}^\infty (1 - p^{-i})}
{p^{r(u+r)} \prod_{i = 1}^{r+u} (1 - p^{-i})}
\]
for nonnegative integers $u$ and $r$:

\begin{conjecture}
\label{C:K-theory conjecture for rank}  
Fix $n \ge 1$ and an odd prime $p$ and $r \ge 0$.
As $F$ ranges over all 
real \textup{(}resp.\ imaginary\textup{)} quadratic fields, 
$\Prob \left( \dim_{\F_p} K_{2n}(\cO_F)^-_p = r \right)$
is given in the following table
by the entry in the row determined by $n$ 
and column determined by the signature:
\begin{center}
\setstretch{1.2}
\resizebox{\linewidth}{!}{
\begin{tabular}{c|c||c|c}
&& \textup{real} & \textup{imaginary} \\ \hline \hline
$\textup{$n$ even}$ & $n \equiv 0 \pmod{p-1}$ & $
\frac{1}{2} (\frac{p}{p+1}  \alpha_{p,2,r-1} + \frac{p+2}{p+1}  \alpha_{p,1,r})
$ & $
\frac{1}{2} (\frac{p}{p+1} \alpha_{p,1,r-1} + \frac{p+2}{p+1} \alpha_{p,0,r})
$
\\ 
& $n \equiv \frac{p-1}{2} \pmod{p-1}$ & $
\frac{1}{2} ( \frac{1}{p+1}\alpha_{p,2,r-1} + \frac{2p+1}{p+1}\alpha_{p,1,r})
$ & $
\frac{1}{2} ( \frac{1}{p+1}\alpha_{p,1,r-1} + \frac{2p+1}{p+1}\alpha_{p,0,r})
$
\\
& $\textup{all other cases}$ & $\alpha_{p,1,r}$ & $\alpha_{p,0,r}$
\\ \hline

$\textup{$n$ odd}$ & $n \equiv \frac{p-1}{2} \pmod{p-1}$ & 
$
\frac{1}{2} ( \frac{1}{p+1}\alpha_{p,1,r-1} + \frac{2p+1}{p+1}\alpha_{p,0,r})
$ & $
\frac{1}{2} ( \frac{1}{p+1}\alpha_{p,2,r-1} + \frac{2p+1}{p+1}\alpha_{p,1,r})
$
\\
&  $\textup{all other cases}$ & $\alpha_{p,0,r}$ & $\alpha_{p,1,r}$. \\ 
\end{tabular}
}
\setstretch{1.0}
\end{center}
\end{conjecture}

To pass from the distribution of $\dim_{\F_p} K_{2n}(\cO_F)^-_p$ 
to that of $\dim_{\F_p} K_{2n}(\cO_F)_p$ itself, 
add the constant $\dim_{\F_p} K_{2n}(\Z)_p$, 
which can be expressed in terms of a class group 
(see Section~\ref{S:even K-groups of Z}).

Conjecture \ref{C:K-theory conjecture for rank} 
implies an average order for $K_{2n}(\cO_F)_p$ 
as $F$ varies over all real or imaginary fields 
(see Conjecture \ref{C:average K-theory conjecture}).  
We prove that this conjectured average order is correct for $p = 3$:

\begin{theorem}
\label{T:average of K_3} 
Fix $n \ge 1$.
The average order of $K_{2n}(\cO_F)_3$
as $F$ ranges over 
real \textup{(}resp.\ imaginary\textup{)} quadratic fields 
is as follows:
\begin{center}
\begin{tabular}{c||c|c}
& $\textup{real}$ & $\textup{imaginary}$ \\ \hline \hline
$\textup{$n$ even}$ & $25/12$ & $11/4$ \\ \hline
$\textup{$n$ odd}$ & $9/4$ & $19/12$. \\ 
\end{tabular}
\end{center}
\end{theorem}

\begin{remark} 
By Theorem \ref{T:K-groups of Z}, $K_{2n}(\ZZ)_3 = 0$ for all $n$ 
since $\Q(\zeta_3)$ has class number~$1$. 
Thus $K_{2n}(\cO_F)^-_3 = K_{2n}(\cO_F)_3$ for all $n$.
\end{remark}

\begin{remark} 
Theorem \ref{T:average of K_3} is an analogue of 
the Davenport--Heilbronn theorem giving the average order of $\Cl(F)_3$ 
as $F$ varies over all real or imaginary quadratic 
fields~\cite{Davenport-Heilbronn1971}*{Theorem~3}.
\end{remark}

\begin{remark}
After this article was written, the second author proved an analogue
of Theorem~\ref{T:average of K_3} for $K_{2n}(\cO_\calF)_2$
as $\calF$ varies over cubic fields \cite{Klagsbrun-preprint2}.
\end{remark}

\subsection{Methods}
The $p$-torsion subgroup $G_p$ of a finite abelian group $G$ 
has the same $\F_p$-dimension as $G/p \colonequals G/pG$;
therefore we study $K_{2n}(\cO_F)/p$.
The latter is isomorphic to an \'etale cohomology group
$H^2_{\et}(\cO_F[1/p],\mu_p^{\tensor (n+1)})$,
which we relate to isotypic components of the class group
and Brauer group of $\cO_E[1/p]$, where $E \colonequals F(\zeta_p)$.
The Brauer group can be computed explicitly,
and we develop heuristics for the class groups;
combining these gives the conjectural distribution of $K_{2n}(\cO_F)_p$.

In the case $p=3$, the isotypic components of $\Cl(\cO_E[1/p])$
are related to $\Cl(\cO_K[1/p])$ for quadratic fields $K$.
The \emph{average order} of the latter class groups can be computed
unconditionally by using a strategy of Davenport and Heilbronn,
which we refine using recent work of Bhargava, Shankar, and Tsimerman,
to control averages in subfamilies
with prescribed local behavior at $3$.
This yields unconditional results on the average order of 
$K_{2n}(\cO_F)_3$.

\subsection{Prior work}
As far as we know, Cohen--Lenstra style conjectures 
have not been proposed for $K$-groups before, 
but some results on the distribution of $K_2(\cO_F)$ have been proved.

Guo~\cite{Guo2009} proved that $4$-ranks of $K_2(\cO_F)$ 
for quadratic fields $F$ follow a Cohen--Lenstra distribution, 
just as Fouvry and Kl\"uners proved for $4$-ranks 
of $\Cl(\cO_F)$ \cite{Fouvry-Klueners2007}.
Studying $4$-ranks is natural,
since the $2$-rank of $K_2(\cO_F)$ for a quadratic field $F$ 
is determined by genus theory 
just as the $2$-rank of $\Cl(\cO_F)$ is 
(see \cite{Browkin-Schinzel1982}, for example).
 
Similar results on the $3$-ranks of $K_2(\cO_L)$ 
for cyclic cubic fields $L$ are due to 
Cheng, Guo, and Qin \cite{Cheng-Guo-Qin2014}.
In addition, Browkin showed that Cohen--Martinet heuristics 
suggest a conjecture on $\Prob(3 \mid \#K_2(\cO_F))$
as $F$ ranges over quadratic fields of fixed signature \cite{Browkin2000}.

\subsection{Notation}
\label{S:notation}

If $G$ is an abelian group and $n \ge 1$,
let $G_n \colonequals \{g \in G:ng=0\}$
and $G/n \colonequals G/nG$.
For any $k$-representation $V$ of a finite group $G$
such that $\Char k \nmid \#G$, 
and for any irreducible $k$-representation $\chi$ of $G$,
let $V^\chi$ be the $\chi$-isotypic component.

Throughout the paper, $p$ is an odd prime, 
$\calF$ is an arbitrary number field,
$\overline{\calF}$ is an algebraic closure of $\calF$,
the element $\zeta_p \in \overline{\calF}$ is a primitive $p$th root of unity,
and $\calE \colonequals \calF(\zeta_p)$.
Later we specialize $\calF$ to a quadratic field $F$
and define $E \colonequals F(\zeta_p)$.

Let $\cO_\calF$ be the ring of integers of $\calF$.
If $S$ is a finite set of places of $\calF$ containing 
all the archimedean places,
define the ring of $S$-integers
$\cO_S := \{x \in \calF: v(x) \ge 0 \textup{ for all $v \notin S$}\}$.
Let $d_\calF \in \Z$ be the discriminant of $\calF$.
Let $\mu(\calF)$ be the group of roots of unity in $\calF$.

If $\cO$ is a Dedekind ring, let $\Cl(\cO)$ denote its class group,
and let $\Br(\cO)$ be its (cohomological) Brauer group, defined as 
$H^2_{\et}(\Spec \cO,\G_m)$ \cite{Poonen2017}*{Definition~6.6.4}.
{}From now on, all cohomology is \'etale cohomology, 
and we drop the subscript $\et$.

\section{From \texorpdfstring{$K$}{K}-theory to class groups and Brauer groups}
\label{S:from K-theory}

In this section, following Tate's argument for $K_2$ \cite{Tate1976}, 
we relate the even $K$-groups
to more concrete groups: class groups and Brauer groups.

\begin{theorem}[Corollary 71 in \cite{Weibel2005}]
\label{T:K-theory and etale cohomology} 
  For any number field $\calF$ and any $n \ge 1$,
  \[
  K_{2n}(\cO_\calF)/p \isom H^2(\cO_\calF[1/p],\mu_p^{\tensor (n+1)}).
  \]
\end{theorem}

\begin{lemma}
  \label{L:H^2 of mu_p}
There is a canonical exact sequence
  \[
  	0 \To \Cl(\cO_\calF[1/p])/p \To H^2(\cO_\calF[1/p],\mu_p) \To \Br(\cO_\calF[1/p])_p \To 0.
  \]
\end{lemma}

\begin{proof}
  Consider the exact sequence
  \[
  	1 \to \mu_p \to \G_m \stackrel{p}\to \G_m \to 1
  \]
  of sheaves on $(\Spec \cO_\calF[1/p])_{\et}$.
Take the associated long exact sequence of cohomology,
and substitute $H^1(\cO_\calF[1/p],\G_m) = \Pic(\cO_\calF[1/p]) = \Cl(\cO_\calF[1/p])$
  and $H^2(\cO_\calF[1/p],\G_m) = \Br(\cO_\calF[1/p])$.
\end{proof}

\begin{lemma}
  \label{L:Hochschild-Serre}
  Let $\calF'/\calF$ be a finite Galois extension of degree prime to $p$.
  Let $i \ge 0$ and $r \in \Z$.
  Then 
  \[
	H^i(\cO_\calF[1/p],\mu_p^{\tensor r}) =  H^i(\cO_{\calF'}[1/p],\mu_p^{\tensor r})^{\Gal(\calF'/\calF)}.
  \]
\end{lemma}

\begin{proof}
  Let $G=\Gal(\calF'/\calF)$.
  Let $X=\Spec \cO_{\calF}[1/p]$ and $X'=\Spec \cO_{\calF'}[1/p]$.
  Let $M = \mu_p^{\tensor r}$.
If $X' \to X$ were \'etale, we would have the Hochschild--Serre spectral sequence
  \[
  	H^i(G,H^j(X',M)) \implies H^{i+j}(X,M),
  \]
and the groups $H^i(G,H^j(X',M))$ for $i>0$
are $0$ because they are killed by both $\#G$ and $\#M=p$,
so the result would follow.

In general, let $U$ be a dense open subscheme of $X$ above which $X' \to X$ is \'etale.
  Let $Z=X-U$.
  Let $U'$ and $Z'$ be the preimages of $U$ and $Z$ under $X' \to X$;
  give $Z$ and $Z'$ the reduced scheme structure.
  The Gysin sequence
  \[
  	\cdots \to H^i_Z(X,M) \to  H^i(X,M) \to H^i(U,M) \to \cdots
  \]
  is exact \cite[Proposition~III.1.25]{Milne1980}.
  Let $N = M(-1) = \mu_p^{\tensor (r-1)}$.
  Purity for closed embeddings of regular schemes replaces $H^i_Z(X,M)$ with $H^{i-2}(Z,N)$ for each $i$; see \cite{Fujiwara2002} and the second paragraph of the proof of \cite[Theorem~VI.5.1]{Milne1980}.
  Taking $G$-invariants of the analogous Gysin sequence for $(X',Z')$ preserves exactness since $p \nmid \#G$.
  Thus we obtain a commutative diagram
  \[
  \xymatrix@C=2ex{
    \cdots \ar[r] & H^{i-1}(U,M) \ar[r] \ar[d]^{\beta_{i-1}} & H^{i-2}(Z,N) \ar[r] \ar[d]^{\gamma_{i-2}} & H^i(X,M) \ar[r] \ar[d]^{\alpha_i} & H^i(U,M) \ar[r] \ar[d]^{\beta_i} & H^{i-1}(Z,N) \ar[r] \ar[d]^{\gamma_{i-1}} & \cdots \\
    \cdots \ar[r] & H^{i-1}(U',M)^G \ar[r] & H^{i-2}(Z',N)^G \ar[r] & H^i(X',M)^G \ar[r] & H^i(U',M)^G \ar[r] & H^{i-1}(Z',N)^G \ar[r] & \cdots 
    }
  \]
  with exact rows.
  
\medskip

  \emph{Claim~1: For every $j$, the map $\beta_j \colon H^j(U,M) \to H^j(U',M)^G$ is an isomorphism.}

  Proof: The morphism $U' \to U$ is a Galois \'etale covering, so the Hochschild--Serre argument applies.

\medskip

\emph{Claim~2: For every $j$, the map $\gamma_j \colon H^j(Z,N) \to H^j(Z',N)^G$ is an isomorphism.}

  Proof: The scheme $Z$ is a finite disjoint union of reduced points.  Fix one, say $\Spec k$.
  Its preimage in $Z'$ is a disjoint union of copies of $\Spec k'$ for some finite Galois extension $k'$ of $k$, and $G$ acts transitively on these copies, so Claim~2 reduces to showing that $H^j(k,N) \to H^j(k',N)^D$ is an isomorphism, where $D$ is the decomposition group.
  Moreover, $D$ acts through $D/I \isom \Gal(k'/k)$, where $I$ is the inertia group, and $\#\Gal(k'/k)$ divides $\#G$ and hence is prime to $p$, so Claim~2 follows from the Hochschild--Serre argument for the \'etale extension $k'/k$.

\medskip

Claims 1 and~2 imply that all vertical maps in the two-row diagram except possibly the middle one are isomorphisms.  By the five lemma, $\alpha_i \colon H^i(X,M) \to H^i(X',M)^G$ is an isomorphism too.
\end{proof}

We now specialize $\calF'$ to $\calE \colonequals \calF(\zeta_p)$.
The action of $\Gal(\calE/\calF)$ on the $p$th roots of $1$
defines an injective homomorphism 
$\chi_1 \colon \Gal(\calE/\calF) \to (\Z/p\Z)^\times$.
For $m \in \Z$, 
composing $\chi_1$ with the $m$th power map on $(\Z/p\Z)^\times$ 
yields another $1$-dimensional $\F_p$-representation of $\Gal(\calE/\calF)$;
call it $\chi_m$.

\begin{lemma}
\label{L:split exact sequence for H^2}
There is a split exact sequence
\[
  	0 \to \left( \Cl(\cO_\calE[1/p])/p \right)^{\chi_{-n}}
	\to H^2(\cO_\calF[1/p],\mu_p^{\tensor (n+1)}) 
	\to \left( \Br(\cO_\calE[1/p])_p  \right)^{\chi_{-n}}
	\to 0.
\]
\end{lemma}

\begin{proof}
First, we have
\begin{align}
\nonumber
	H^2(\cO_\calF[1/p],\mu_p^{\tensor (n+1)}) 
	&= H^2(\cO_\calE[1/p],\mu_p^{\tensor (n+1)})^{\Gal(\calE/\calF)} 
	\quad\textup{(Lemma~\ref{L:Hochschild-Serre})} \\
\nonumber
	&= \left( H^2(\cO_\calE[1/p],\mu_p) 
		\tensor \mu_p^{\tensor n} \right)^{\Gal(\calE/\calF)} 
	\quad\textup{(since $\mu_p \subset \calE$)}\\
\label{E:H^2 for E and F}
	&= H^2(\cO_\calE[1/p],\mu_p)^{\chi_{-n}} 
	\quad\textup{(since $\mu_p^{\tensor n} \isom \chi_n$).}
\end{align}
On the other hand,
Lemma~\ref{L:H^2 of mu_p} for $\calE$ yields
a sequence of $\Gal(\calE/\calF)$-representations
\[
  	0 \To \Cl(\cO_\calE[1/p])/p 
	\To H^2(\cO_\calE[1/p],\mu_p) 
	\To \Br(\cO_\calE[1/p])_p 
	\To 0,
\]
which splits by Maschke's theorem.
Take $\chi_{-n}$-isotypic components, and 
substitute \eqref{E:H^2 for E and F} in the middle.
\end{proof}

Substituting Theorem~\ref{T:K-theory and etale cohomology}
into Lemma~\ref{L:split exact sequence for H^2}
yields the main result of this section:

\begin{theorem}
\label{T:K_2n in terms of Cl and Br}
For each $n \ge 1$,
\[
	K_{2n}(\cO_\calF)/p \isom 
	\left( \Cl(\cO_\calE[1/p])/p \right)^{\chi_{-n}}
	\directsum
	\left( \Br(\cO_\calE[1/p])_p  \right)^{\chi_{-n}}.
\]
\end{theorem}

\section{Even \texorpdfstring{$K$}{K}-groups of the ring of integers of a quadratic field}
\label{S:K-groups of quadratic field}

Let $p^* = (-1)^{(p-1)/2} p$,
so $\Q(\sqrt{p^*})$ is the degree~$2$ subfield of $\Q(\zeta_p)$.
{}From now on, 
$F$ is a degree~$2$ extension of $\Q$ not equal to $\Q(\sqrt{p^*})$.
Thus $F=\Q(\sqrt{d})$ for some $d \in \Q^\times$
such that $d$ and $p^*$ are independent in $\Q^\times/\Q^{\times 2}$.
Let $E=F(\zeta_p)$.
Then
\[
	\Gal(E/\Q) \isom (\Z/p\Z)^\times \times \{\pm 1\}.
\]
Let $\tau$ be the generator of $\Gal(E/\Q(\zeta_p)) = \{\pm 1\}$,
so $\tau$ restricts to the generator of $\Gal(F/\Q)$.
The action of $\tau$ decomposes $K_{2n}(\cO_F)/p$
into $+$ and $-$ eigenspaces.
Let $\chi_{-n,-1} \colon G \to \F_p^\times$
be such that $\chi_{-n,-1}|_{\Gal(E/F)} = \chi_{-n}$
and $\chi_{-n,-1}(\tau)=-1$.

\begin{theorem}
\label{T:K_2n for quadratic field}
For each $n \ge 1$,
\begin{align*}
	(K_{2n}(\cO_F)/p)^+ &\isom K_{2n}(\Z)/p \\
	(K_{2n}(\cO_F)/p)^- &\isom 
	\left( \Cl(\cO_E[1/p])/p \right)^{\chi_{-n,-1}}
	\directsum
	\left( \Br(\cO_E[1/p])_p  \right)^{\chi_{-n,-1}}.
\end{align*}
\end{theorem}

\begin{proof}
To obtain the first statement, use 
Theorem~\ref{T:K-theory and etale cohomology} 
to rewrite each term as an \'etale cohomology group
and apply Lemma~\ref{L:Hochschild-Serre} with $\calF'/\calF$ there being $F/\Q$.
To obtain the second,
take minus parts in Theorem~\ref{T:K_2n in terms of Cl and Br}.
\end{proof}

\section{Brauer groups}
\label{S:Brauer groups}

The goal of this section is to determine the rightmost term 
in Theorem~\ref{T:K_2n for quadratic field}.

\begin{lemma}[Equation (6.9.5) in \cite{Poonen2017}]
  Let $\calF$ and $\cO_S$ be as in Section~\ref{S:notation}.
  Let $r_1$ be the number of real places of $\calF$.
Then there is an exact sequence
\[
	0 \To \Br \cO_S \To \left(\tfrac12 \Z/\Z \right)^{r_1} \directsum \Directsum_{\textup{finite $v \in S$}} \Q/\Z \stackrel{\textup{sum}}\To \Q/\Z.
\]
\end{lemma}

\begin{corollary}
\label{C:p-torsion in Br}
We have $\Br(\cO_S)_p \isom (\Z/p\Z)^{\{\textup{finite $v \in S$}\}}_{\textup{sum 0}}$,
  where the ``sum~$0$'' subscript denotes the subgroup of elements whose sum is $0$.
\end{corollary}

\begin{corollary}
\label{C:Br Z[zeta_p,1/p]}
We have $\Br(\Z[\zeta_p,1/p])_p = 0$.
\end{corollary}

\begin{proof}
There is only one prime above $p$ in $\Z[\zeta_p]$.
\end{proof}

\begin{proposition}
\label{P:p-torsion in Br for quadratic fields}
Let $F$ and $E$ be as in Section~\ref{S:K-groups of quadratic field}.
Then
\[
(\Br(\cO_E[1/p])_p)^{\chi_{-n,-1}} = 
\begin{cases}
\Z/p\Z, & \textup{if $n \equiv 0 \!\!\!\pmod{p-1}$ and $d \in \Q_p^{\times 2}$;} \\
\Z/p\Z, & \textup{if $n \equiv \dfrac{p-1}{2} \!\!\!\pmod{p-1}$ and $p^* d \in \Q_p^{\times 2}$;} \\
0, & \textup{in all other cases.}
\end{cases}
\]
\end{proposition}

\begin{proof}
The hypothesis implies that $F$
is not the quadratic subfield $\Q(\sqrt{p^*})$ of $\Q(\zeta_p)$.
Thus $E$ is the compositum of linearly disjoint extensions
$\Q(\zeta_p)$ and $F$ over $\Q$, 
and $\Gal(E/\Q) \isom (\Z/p\Z)^\times \times \{\pm 1\}$.
Let $\p$ be a prime of $E$ lying above $p$.
Let $D \le \Gal(E/\Q)$ be the decomposition group of $\p$.
Since $p$ totally ramifies in $\Q(\zeta_p)/\Q$,
we have $p-1 \mid \#D$.
Let $S_p$ be the set of primes of $E$ lying above $p$,
so $S_p \isom \Gal(E/\Q)/D$, 
which by the previous sentence is of size $1$ or $2$;
it is $2$ if and only if
$p$ splits in one of the quadratic subfields of $E$.
These quadratic subfields are $\Q(\sqrt{p^*})$, $F$, and 
the field $F' = \Q(\sqrt{p^* d})$,
but $p$ is ramified in $\Q(\sqrt{p^*})$.
Thus by Corollary~\ref{C:p-torsion in Br},
\[
\Br(\cO_E[1/p])_p = (\Z/p\Z)^{S_p}_{\textup{sum $0$}} = 
\begin{cases}
\Z/p\Z, & \textup{if $p$ splits in $F$ or $F'$;}\\
0, & \textup{otherwise}
\end{cases}
\]
as an abelian group, and it remains to determine 
in the first case which character it is isomorphic to.
We will compute the action of $\Gal(E/F) \isom (\Z/p\Z)^\times$
and $\Gal(E/\Q(\zeta_p)) \isom \{\pm1\}$ separately.

If $p$ splits in $F$ (that is, $d \in \Q_p^{\times 2}$), 
then $D = \Gal(E/F)$, 
which acts trivially on $S_p \isom \Gal(E/\Q)/D$, 
so $\Gal(E/F)$ acts on $\Br(\cO_E[1/p])_p$ as the trivial character $\chi_0$.
If instead $p$ splits in $F'$ (that is, $p^* d \in \Q_p^{\times 2}$), 
then $D \ne \Gal(E/F)$,
so $\Gal(E/F)$ acts nontrivially on the two-element set 
$S_p \isom \Gal(E/\Q)/D$, 
so $\Gal(E/F)$ acts on $\Br(\cO_E[1/p])_p$ 
as the character $\Gal(E/F) \surjects \{\pm 1\}$,
which is $\chi_{(p-1)/2}$.

Finally, consider the action of the generator $\tau$
of $\Gal(E/\Q(\zeta_p)) \isom \{\pm1\}$ on $\Br(\cO_E[1/p])_p$.
Lemma~\ref{L:Hochschild-Serre} shows that the $+$ eigenspace
is $\Br(\Z[\zeta_p,1/p])_p$, 
which is $0$ by Corollary~\ref{C:Br Z[zeta_p,1/p]}.
Thus $\Br(\cO_E[1/p])_p$ equals its $-$ eigenspace.
\end{proof}

\begin{remark}
\label{R:splits}
Let $d$ be a fundamental discriminant.
Then $d \in \Q_p^{\times 2}$ if and only if $\Jacobi{d}{p}=1$,
and $p^*d \in \Q_p^{\times 2}$ if and only if
$p|d$ and $\Jacobi{-d/p}{p}=1$.
\end{remark}

\section{Even \texorpdfstring{$K$}{K}-groups of \texorpdfstring{$\Z$}{Z}}
\label{S:even K-groups of Z}

\begin{theorem}
\label{T:K-groups of Z}
For each $n \ge 1$,
\[
	K_{2n}(\Z)/p \isom \left(\Cl(\Z[\zeta_p])/p\right)^{\chi_{-n}}.
\]
\end{theorem}

\begin{proof}
In Theorem~\ref{T:K_2n in terms of Cl and Br} for $\calF=\Q$,
the Brauer term is $0$ by Corollary~\ref{C:Br Z[zeta_p,1/p]},
and $\Cl(\Z[\zeta_p,1/p]) = \Cl(\Z[\zeta_p])$
since the unique prime ideal above $p$ in $\Z[\zeta_p]$ is principal.
\end{proof}

For $n \ge 1$, let 
\[
	\kappa_{2n,p} 
	\colonequals \dim_{\F_p} K_{2n}(\Z)/p
	= \dim_{\F_p} \left(\Cl(\Z[\zeta_p])/p\right)^{\chi_{-n}}.
\]

\begin{remark}
Assuming Vandiver's conjecture that $p \nmid \#\Cl(\Z[\zeta_p+\zeta_p^{-1}])$
for every prime $p$, the $K$-groups of $\Z$ are known; 
see \cite{Weibel2005}*{Section~5.9}.
To state the results for even $K$-groups,
let $B_{2k} \in \Q$ be the $(2k)$th Bernoulli number,
defined by
\[
\frac{t}{e^t-1}=1-\frac{t}{2}+\sum_{k=1}^\infty B_{2k} \frac{t^{2k}}{(2k)!} .
\]
Let $c_k$ be the numerator of $|B_{2k}/(4k)|$.
Then (see \cite{Weibel2005}*{Corollary~107})
\begin{itemize}
\item Vandiver's conjecture implies that $K_{4k}(\Z)=0$ for all $k \ge 1$.
\item For $k \ge 1$, 
the order of $K_{4k-2}(\Z)$ is $c_k$ if $k$ is even, and $2c_k$
if $k$ is odd; moreover, Vandiver's conjecture implies that $K_{4k-2}(\Z)$ 
is cyclic.
\end{itemize}
In fact, for each prime $p$, Vandiver's conjecture for $p$ 
implies the conclusions above for the $p$-primary part of the $K$-groups.
Thus Vandiver's conjecture for an odd prime $p$ 
implies that for any $n \ge 1$, 
\[
	\kappa_{2n,p} = 
	\begin{cases}
		1 & \textup{ if $n=2k-1$ and $p|c_k$;} \\
		0 & \textup{ otherwise.}
	\end{cases}
\]
Moreover, Vandiver's conjecture is known 
for $p<163577856$ \cite{Buhler-Harvey2011}.

The smallest odd prime $p$ for which there exists $n$
such that $p | \#K_{2n}(\Z)$ is the smallest irregular prime, $37$, 
which divides $\#K_{2n}(\Z)$
if and only if $n \equiv 31 \pmod{36}$;
thus $\kappa_{2n,37}$ is $1$ if $n \equiv 31 \pmod{36}$, and $0$ otherwise.
Assuming Vandiver's conjecture, 
the smallest $n$ such that $\#K_{2n}(\Z)$ is divisible by an odd prime
is $n=11$: we have $K_{22}(\Z) \isom \Z/691\Z$.
See \cite{Weibel2005}*{Example~96} for these and other examples.
\end{remark}

\section{Odd \texorpdfstring{$K$}{K}-groups}
\label{sec:odd K groups}

\begin{proposition}
For any number field $\calF$, positive integer $i$, and odd prime $p$, 
we have
\[
K_{2i-1}(\cO_\calF)_p = 
\begin{cases}
	\Z/p\Z, & \textup{ if $[\calF(\zeta_p):\calF]$ divides $i$;} \\
	0, & \textup{ otherwise.}
\end{cases}
\]
\end{proposition}

\begin{proof}
Define the group
\[
	\mu^{(i)}(\calF) \colonequals\{\zeta \in \mu(\overline{\calF}) 
	: \sigma^i \zeta = \zeta
	\textup{ for all $\sigma \in \Gal(\overline{\calF}/\calF)$}\}.
\]
For $n \ge 1$, let $\zeta_n \in \overline{\calF}$ be a primitive $n$th root
of $1$, and let $H_n$ be the image of the restriction homomorphism
$\Gal(\calF(\zeta_n)/\calF) \injects \Gal(\Q(\zeta_n)/\Q) \isom (\Z/n\Z)^\times$,
so $\#H_n=[\calF(\zeta_n):\calF]$.
Then the following are equivalent:
\begin{itemize}
\item $\zeta_n \in \mu^{(i)}(\calF)$;
\item $\sigma^i \zeta_n \equiv \zeta_n \text{ for all } \sigma \in \Gal(\overline{\calF}/\calF)$;
\item $a^i = 1$ for all $a \in H_n$.
\end{itemize}
Now suppose that $n$ is a prime power $\ell^e$ for some prime $\ell$.
Then $H_n$ contains a cyclic subgroup of index at most~$2$ 
(we allow the case $\ell=2$).
The last condition above implies $\#H_n | 2i$,
which after multiplication by $[\calF:\Q]$ becomes the statement that
$[\calF(\zeta_n):\Q]$ divides $2i[\calF:\Q]$,
which implies that the integer $\phi(n) \colonequals [\Q(\zeta_n):\Q]$
divides $2i[\calF:\Q]$,
which bounds $\phi(n)$ and hence $n$.
Thus $\mu^{(i)}(\calF)$ contains $\ell^e$th roots of $1$
for only finitely many prime powers $\ell^e$,
so it is finite.
Define $w^{(i)}(\calF) \colonequals \# \mu^{(i)}(\calF)$.

By Theorem~70 in \cite{Weibel2005}, if $p$ is an odd prime,
$K_{2i-1}(\cO_\calF)_p$ is $\Z/p\Z$ or $0$,
according to whether $p$ divides $w^{(i)}(\calF)$ or not.
The previous paragraph shows that the latter condition
is equivalent to $H_p$ being killed by $i$,
and to $\#H_p | i$ since $H_p$ is cyclic
(a subgroup of the cyclic group $(\Z/p\Z)^\times$).
Finally, $\#H_p = [\calF(\zeta_p):\calF]$.
\end{proof}

\section{Heuristics for class groups}
\label{S:heuristics for class groups}

Let $A$ be an abelian extension of $\Q$.
Suppose that the Galois group $G \colonequals \Gal(A/\Q)$
is of exponent dividing $p-1$.
Let $I^p_A$ be the group of fractional ideals of $\cO_A[1/p]$,
that is, the free abelian group on the set of finite primes of $A$
not lying above $p$.
We have the standard exact sequence of $\Z G$-modules
\begin{equation}
\label{E:4-term}
	1 \to \cO_A[1/p]^\times \to A^\times \to I^p_A \to \Cl(\cO_A[1/p]) \to 0.
\end{equation}

Let $S$ be a finite set of places of $\Q$ including $p$ and $\infty$.
Let $S_A$ be the set of places of $A$ above $S$.
We approximate \eqref{E:4-term} by using only $S_A$-units
and ideals supported on $S_A$.
Let $S_A^p$ be the set of places of $A$ above $S-\{p\}$.
Let $\cO_{A,S}$ be the ring of $S_A$-integers in $A$.
Let $I^p_{A,S}$ be the free abelian group on the nonarchimedean
places in $S_A^p$.
If $S$ is large enough that the finite primes in $S_A^p$ generate
$\Cl(\cO_A[1/p])$,
then we have an exact sequence of $\Z G$-modules
\begin{equation}
\label{E:approximate 4-term}
	1 \to \cO_A[1/p]^\times \to \cO_{A,S}^\times \to I^p_{A,S} \to \Cl(\cO_A[1/p]) \to 0.
\end{equation}
Dropping the first term and tensoring with $\F_p$ 
yields an exact sequence of $\F_p G$-modules
\[
	\cO_{A,S}^\times/p \to I^p_{A,S}/p \to \Cl(\cO_A[1/p])/p \to 0.
\]
Let $\chi$ be an irreducible $\F_p$-representation of $G$;
our assumption on $G$ guarantees that $\chi$ is $1$-dimensional.
Taking $\chi$-isotypic components yields
\begin{equation}
\label{E:class group as cokernel}
	\left(\cO_{A,S}^\times/p \right)^\chi 
	\to \left( I^p_{A,S}/p \right)^\chi 
	\to \left( \Cl(\cO_A[1/p])/p \right)^\chi 
	\to 0.
\end{equation}
Let 
$u = u(A,\chi) \colonequals \dim_{\F_p} \left(\cO_A[1/p]^\times/p \right)^\chi$.

\begin{lemma}
\label{L:u-lemma}
Assume that $\mu_p(A)^\chi=0$.
\begin{enumerate}[\upshape (a)]
\item \label{E:u-formula}
Let $S_\infty$ \textup{(}resp.\ $S_p$\textup{)} 
be the set of places of $A$ lying above $\infty$ \textup{(}resp.\ $p$\textup{)}.
Then 
\[
	u = \dim_{\F_p} (\F_p^{S_\infty})^\chi 
		+ \dim_{\F_p} (\F_p^{S_p})^\chi 
	- \begin{cases}
		1, &\textup{if $\chi=1$;} \\ 
		0, &\textup{otherwise.}
	\end{cases}
\]
\item 
We have 
\[
	\dim_{\F_p} \left(\cO_{A,S}^\times/p \right)^\chi 
	= \dim_{\F_p} \left(I^p_{A,S}/p \right)^\chi  + u.
\]
\item 
The quantity $\dim_{\F_p} \left(I^p_{A,S}/p \right)^\chi$
can be made arbitrarily large by choosing $S$ appropriately.
\end{enumerate}
\end{lemma}

\begin{proof}\hfill
\begin{enumerate}[\upshape (a)]
\item 
The Dirichlet $S$-unit theorem implies that the abelian group
$\cO_A[1/p]^\times$ is finitely generated with torsion subgroup $\mu(A)$.
Let $M$ be the $\Z G$-module $\cO_A[1/p]^\times/\mu(A)$.
Tensoring the exact sequence
\[
	\mu(A) \To \cO_A[1/p]^\times \To M \To 0
\]
with $\F_p$ and taking $\chi$-isotypic components yields
\[
	0 \To \left( \cO_A[1/p]^\times/p \right)^\chi \To (M/p)^\chi \To 0,
\]
so $u = \dim_{\F_p} (M/p)^\chi$.

On the other hand, the proof of the Dirichlet $S$-unit theorem yields
\[
	M \tensor \R 
	\isom \cO_A[1/p]^\times \tensor \R 
	\isom \left(\R^{S_\infty \union S_p}\right)_{\textup{sum~$0$}}
\]
as $\R G$-modules.
A $\Z_{(p)} G$-module that is free of finite rank over $\Z_{(p)}$
is determined by its character, so
\[
	M \tensor \Z_{(p)} 
	\isom \left( \Z_{(p)}^{S_\infty \union S_p}\right)_{\textup{sum~$0$}}
\]
as $\Z_{(p)} G$-modules.
Both sides are free over $\Z_{(p)}$, so we may tensor with $\F_p$ to obtain
\[
	M/p \isom \left( \F_p^{S_\infty \union S_p}\right)_{\textup{sum~$0$}}
\]
as $\F_p G$-modules.
In other words, there is an exact sequence
\[
0 \To M/p \To \F_p^{S_\infty} \directsum \F_p^{S_p} \To \F_p \To 0.
\]
Taking dimensions of the $\chi$-components yields the formula for $u$.
\item 
The composition $G \stackrel{\chi}\to (\Z/p\Z)^\times \injects \Z_p^\times$ 
lets us identify $\chi$ with a $\Z_p$-representation of $G$.
Tensor \eqref{E:approximate 4-term} with $\Z_p$,
and take $\chi$-isotypic components:
\[
	0 \to \left( \cO_A[1/p]^\times \tensor \Z_p \right)^\chi 
	\to \left( \cO_{A,S}^\times \tensor \Z_p \right)^\chi 
	\to \left( I^p_{A,S} \tensor \Z_p \right)^\chi 
	\to \left( \Cl(\cO_A[1/p]) \tensor \Z_p \right)^\chi 
	\to 0.
\]
Since $\mu_p(A)^\chi=0$, the first three $\Z_p$-modules are free;
on the other hand, the last is finite as a set.
Take $\Z_p$-ranks.
If $V$ is a $\Z_p G$-module such that $V^\chi$ is a free $\Z_p$-module
of finite rank,
then $\dim_{\F_p} (V/p)^\chi = \rank_{\Z_p} V^\chi$.
This proves the formula.
\item

If $S$ contains $m$ rational primes that split completely in $A$,
then $I^p_{A,S}$ contains $(\Z G)^m$,
so $\dim_{\F_p} \left(I^p_{A,S}/p \right)^\chi \ge m$,
and $m$ can be chosen arbitrarily large.\qedhere
\end{enumerate}
\end{proof}

Sequence~\eqref{E:class group as cokernel}
and Lemma~\ref{L:u-lemma}(b,c) imply that 
$\left( \Cl(\cO_A[1/p])/p \right)^\chi$
is naturally the cokernel of a linear map $\F_p^{m+u} \to \F_p^m$
for arbitrarily large $m$.
In Section~\ref{S:heuristics for even K-groups},
we will vary $(A,\chi)$ in a family with constant $u$-value
and conjecture that 
the distribution of $\left( \Cl(\cO_A[1/p])/p \right)^\chi$
equals the 
limit as $m \to \infty$ of the distribution of the cokernel of 
a \emph{random} linear map $\F_p^{m+u} \to \F_p^m$;
the precise statement is Conjecture~\ref{C:class group heuristic}.
For now, we mention that this limiting distribution 
and the limiting expected size of the cokernel 
are known: 

\begin{proposition}
\label{P:alpha distribution}
Fix a prime $p$ and an integer $u \ge 0$.
For $m \ge 0$, 
let $N$ be a linear map $\F_p^{m+u} \to \F_p^m$ chosen uniformly at random,
and let $\calN_{p,u,m}$ be the random variable $\dim_{\F_p} \coker(N)$.
Then
\begin{enumerate}[\upshape (a)] 
\item \label{I:alpha formula}
For each $r \ge 0$,
\[
\lim_{m \rightarrow \infty} \Prob(\calN_{p,u,m}=r) 
	\; = \; \alpha_{p,u,r} \colonequals 
\frac{\prod_{i = r+1}^\infty (1 - p^{-i})}
     {p^{r(u+r)} \prod_{i = 1}^{r+u} (1 - p^{-i})}.
\]
\item \label{I:alpha probability}
We have $\sum_{r = 0}^\infty \alpha_{p,u,r} = 1$.
\item \label{I:alpha average} 
We have $\sum_{r = 0}^\infty p^r \alpha_{p,u,r} = 1+p^{-u}$.
\end{enumerate}
\end{proposition}

\begin{proof}\hfill
\begin{enumerate}[\upshape (a)]
\item 
This is \cite{Kovalenko-Levitskaja1975a}*{Theorem~1}.
\item 
This is the $q=1/p$ and $\alpha=0$ case 
of \cite{Cohen-Lenstra1984}*{Corollary~6.7}.
\item 
This is the $q=1/p$ and $\alpha=1$ case 
of \cite{Cohen-Lenstra1984}*{Corollary~6.7}.\qedhere
\end{enumerate}
\end{proof}

\begin{remark}
The constant $\alpha_{p,u,r}$ appeared also 
in \cite{Cohen-Lenstra1984}*{Theorem~6.3}, 
as the $u$-probability that a random finite abelian $p$-group has $p$-rank $r$.
The connection between $u$-probabilities and coranks of random matrices 
was made in~\cite{Friedman-Washington1989}.
\end{remark}

\section{Heuristics for class groups and even \texorpdfstring{$K$}{K}-groups associated to quadratic fields}
\label{S:heuristics for even K-groups}

\subsection{Calculation of \texorpdfstring{$u$}{u}}

We now specialize Section~\ref{S:heuristics for class groups}
to the setting of Section~\ref{S:K-groups of quadratic field}.
Thus $F$ is $\Q(\sqrt{d})$ for some $d \in \Q^\times$
such that $d$ and $p^*$ are independent in $\Q^\times/\Q^{\times 2}$;
by multiplying $d$ by a square, we may assume that 
$d$ is a fundamental discriminant: $d=d_F$.
Also, $E \colonequals F(\zeta_p)$ and $\chi \colonequals \chi_{-n,-1}$.
Define $u(E,\chi)$ as in the sentence before Lemma~\ref{L:u-lemma}.

\begin{proposition}
\label{P:u-values for quadratic fields}
The value $u(E,\chi)$ is given by the following table:
\begin{center}
\begin{tabular}{c|c||c|c}
&& $d>0$ & $d<0$ \\ \hline \hline
$\textup{$n$ even}$ & $n \equiv 0 \pmod{p-1}$, $d \in \Q_p^{\times 2}$ & $2$ & $1$ \\ 
& $n \equiv \frac{p-1}{2} \pmod{p-1}$, $p^* d \in \Q_p^{\times 2}$ & $2$ & $1$ \\ 
& $\textup{all other cases}$ & $1$ & $0$ \\ \hline
$\textup{$n$ odd}$ & $n \equiv \frac{p-1}{2} \pmod{p-1}$, $p^* d \in \Q_p^{\times 2}$ & $1$ & $2$ \\ 
&  $\textup{all other cases}$ & $0$ & $1$ \\ 
\end{tabular}
\end{center}
\end{proposition}

\begin{proof}
Because $\Gal(E/\Q(\zeta_p))$ acts on $\mu_p(E)$ as $+1$, 
we have $\mu_p(E)^\chi=0$,
so Lemma~\ref{L:u-lemma}\eqref{E:u-formula} applies with $\chi = \chi_{-n,-1}$.
The complex conjugation in $\Gal(E/\Q) \isom (\Z/p\Z)^\times \times \{\pm1\}$
is $c \colonequals (-1,\pm 1)$, where the $\pm 1$ is $+1$ if $F$ is real,
and $-1$ if $F$ is imaginary.
The $G$-set $S_\infty$ is isomorphic to $G/\langle c \rangle$,
so $\F_p^{S_\infty}$ is the $c$-invariant subrepresentation of 
the regular representation $\F_p G$.
The multiplicity of $\chi$ in $\F_p G$ is $1$,
so the multiplicity of $\chi$ in $\F_p^{S_\infty}$ is $1$ or $0$,
according to whether $\chi(c)$ is $1$ or $-1$.
By definition of $\chi_{-n,-1}$, we have $\chi(c)=(-1)^{-n} \sgn(d)$,
so
\[
	\dim_{\F_p} \left(\F_p^{S_\infty}\right)^{\chi}
	= \begin{cases}
		1, &\textup{if $(-1)^n d > 0$;} \\
		0, &\textup{if $(-1)^n d < 0$.} \\
	\end{cases}
\]
Next, Corollary~\ref{C:p-torsion in Br} implies
\[
	\dim_{\F_p} \left(\F_p^{S_p}\right)^\chi 
	= \dim_{\F_p} (\Br(\cO_E[1/p])_p)^\chi,
\]
which is given by Proposition~\ref{P:p-torsion in Br for quadratic fields}.
The third term in Lemma~\ref{L:u-lemma}\eqref{E:u-formula} is $0$
since $\chi \ne 1$.
\end{proof}

\subsection{Distribution}

Suppose that $\scriptF$ is a family of quadratic fields.
For $X>0$, let $\scriptF_{<X}$ be the set of $F \in \scriptF$
such that $|d_F| < X$.
For any function $\gamma \colon \scriptF \to \Z_{\ge 0}$,
define the following notation: 
\begin{align*}
	\Prob(\gamma(F) = r) &\colonequals 
	\lim_{X \to \infty} \frac{\# \{F \in \scriptF_{<X} : \gamma(F) = r \}}
				{\# \scriptF_{<X}}, \\
	\Average(\gamma) &\colonequals
	\lim_{X \to \infty} \frac{\sum_{F \in \scriptF_{<X}} \gamma(F) }
				{\# \scriptF_{<X}}.
\end{align*}

Our heuristic is formalized in the following statement.

\begin{conjecture}[Distribution of class group components]
\label{C:class group heuristic}
Fix one of the ten boxes \textup{(}below and right of the double lines\textup{)}
in the table of Proposition~\ref{P:u-values for quadratic fields}, 
and fix a corresponding $n \ge 1$.
Let $F$ vary over the family $\scriptF$ of quadratic fields
with $d$ satisfying the conditions defining that box,
and let $u$ be as calculated in 
Proposition~\ref{P:u-values for quadratic fields}.
Then the distribution of $\left( \Cl(\cO_E[1/p])/p \right)^\chi$
equals the limit as $m \to \infty$
of the distribution of a cokernel of a random linear map
$\F_p^{m+u} \to \F_p^m$; by this we mean,
in the notation of Proposition~\ref{P:alpha distribution},
that for each $r \in \Z_{\ge 0}$,
\[
   \Prob \left( \dim_{\F_p} \left( \Cl(\cO_E[1/p])/p \right)^\chi = r \right)
   = \lim_{m \rightarrow \infty} \Prob(\calN_{p,u,m}=r),
\]
which by Proposition~\ref{P:alpha distribution}\eqref{I:alpha formula} 
equals $\alpha_{p,u,r}$.
\end{conjecture}

If Conjecture~\ref{C:class group heuristic} holds,
then substituting it and 
Proposition~\ref{P:p-torsion in Br for quadratic fields}
(with Remark~\ref{R:splits})
into Theorem~\ref{T:K_2n for quadratic field}
yields the following:

\begin{conjecture}[Distribution of $K$-groups modulo $p$ in residue classes]
\label{C:K-theory conjecture}
Fix $n \ge 1$ and an odd prime $p$ and $r \ge 0$.
As $F$ ranges over the quadratic fields $\Q(\sqrt{d})$ 
with $d$ satisfying the conditions defining a box below,
$\Prob \left( \dim_{\F_p} (K_{2n}(\cO_F)/p)^-  = r \right )$
is as follows:
\begin{center}
\begin{tabular}{c|c||c|c}

&& $d>0$ & $d<0$ \\ \hline \hline
$\textup{$n$ even}$ & $n \equiv 0 \pmod{p-1}$, $d \in \Q_p^{\times 2}$ & $\alpha_{p,2,r-1}$ & $\alpha_{p,1,r-1}$ \\ 
& $n \equiv \frac{p-1}{2} \pmod{p-1}$, $p^*d \in \Q_p^{\times 2}$ & $\alpha_{p,2,r-1}$ & $\alpha_{p,1,r-1}$ \\ 
& $\textup{all other cases}$ & $\alpha_{p,1,r}$ & $\alpha_{p,0,r}$ \\ \hline
$\textup{$n$ odd}$ & $n \equiv \frac{p-1}{2} \pmod{p-1}$, $p^* d \in \Q_p^{\times 2}$ & $\alpha_{p,1,r-1}$ & $\alpha_{p,2,r-1}$ \\ 
&  $\textup{all other cases}$ & $\alpha_{p,0,r}$ & $\alpha_{p,1,r}$ \\ 
\end{tabular}
\end{center}
The distribution of $\dim_{\F_p} K_{2n}(\cO_F)/p$ is the distribution of 
$\dim_{\F_p} (K_{2n}(\cO_F)/p)^-$ shifted by 
the constant $\kappa_{2n,p} = \dim_{\F_p} K_{2n}(\Z)_p$
of Section~\textup{\ref{S:even K-groups of Z}}.
\end{conjecture}

In Conjecture~\ref{C:K-theory conjecture},
the family of fields was defined by specifying both the sign of $d$
and a $p$-adic condition on $d$.
To get the analogous probabilities for a larger family 
in which only the sign of $d$ is specified,
we can take a weighted combination of probabilities 
from Conjecture~\ref{C:K-theory conjecture}.

\begin{example}
\label{Ex:upper left}
Suppose that $n \equiv 0 \pmod{p-1}$
and $\scriptF$ is the family of real quadratic fields $F$.
By the first and third entries of the $d>0$ column 
of the table of Conjecture~\ref{C:K-theory conjecture}, 
$\Prob \left( \dim_{\F_p} (K_{2n}(\cO_F)/p)^- = r \right)$
equals
\begin{equation}
\label{E:weighted probability}
	\Prob(d \in \Q_p^{\times 2}) \; \alpha_{p,2,r-1}
	\; + \; 
	\Prob(d \notin \Q_p^{\times 2}) \; \alpha_{p,1,r}.
\end{equation}
Since $p^2 \nmid d$, we have $d \in \Q_p^{\times 2}$
if and only if $(d \bmod p) \in \F_p^{\times 2}$;
this condition is satisfied by $d$ lying in $p \cdot (p-1)/2$ of the $p^2-1$
nonzero residue classes modulo $p^2$.
The discriminant $d$ is equally likely to be in any of the $p^2-1$ nonzero
residue classes modulo $p^2$, 
as follows, for example, from \cite{Prachar1958}*{(1)},
so 
\[
	\Prob(d \in \Q_p^{\times 2}) 
	= \frac{p(p-1)/2}{p^2-1} 
        = \frac{p}{2p+2}.
\]
Substituting this and the complementary probability 
into~\eqref{E:weighted probability} yields 
\[
	\Prob \left( \dim_{\F_p} (K_{2n}(\cO_F)/p)^- = r \right)
	= \frac{p}{2p+2}  \alpha_{p,2,r-1} + \frac{p+2}{2p+2}  \alpha_{p,1,r}.
\]
\end{example}

Similar calculations show that 
Conjecture~\ref{C:K-theory conjecture} implies
all ten cases in Conjecture~\ref{C:K-theory conjecture for rank again} below.

\begin{conjecture}[Distribution of $K$-groups modulo $p$; cf.~Conjecture~\ref{C:K-theory conjecture for rank}]
\label{C:K-theory conjecture for rank again}
Fix $n \ge 1$ and an odd prime $p$ and $r \ge 0$.
As $F$ ranges over all 
real \textup{(}resp.\ imaginary\textup{)} quadratic fields, 
$\Prob \left( \dim_{\F_p} (K_{2n}(\cO_F)/p)^- = r \right)$
is given in the following table
by the entry in the row determined by $n$ 
and column determined by the signature:
\begin{center}
\setstretch{1.2}
\resizebox{\linewidth}{!}{
\begin{tabular}{c|c||c|c}
&& \textup{real} & \textup{imaginary} \\ \hline \hline
$\textup{$n$ even}$ & $n \equiv 0 \pmod{p-1}$ & $
\frac{p}{2p+2}  \alpha_{p,2,r-1} + \frac{p+2}{2p+2}  \alpha_{p,1,r}
$ & $
\frac{p}{2p+2} \alpha_{p,1,r-1} + \frac{p+2}{2p+2} \alpha_{p,0,r}
$
\\ 
& $n \equiv \frac{p-1}{2} \pmod{p-1}$ & $
\frac{1}{2p+2}\alpha_{p,2,r-1} + \frac{2p+1}{2p+2}\alpha_{p,1,r}
$ & $
\frac{1}{2p+2}\alpha_{p,1,r-1} + \frac{2p+1}{2p+2}\alpha_{p,0,r}
$
\\
& $\textup{all other cases}$ & $\alpha_{p,1,r}$ & $\alpha_{p,0,r}$
\\ \hline

$\textup{$n$ odd}$ & $n \equiv \frac{p-1}{2} \pmod{p-1}$ & 
$
\frac{1}{2p+2}\alpha_{p,1,r-1} + \frac{2p+1}{2p+2}\alpha_{p,0,r}
$ & $
\frac{1}{2p+2}\alpha_{p,2,r-1} + \frac{2p+1}{2p+2}\alpha_{p,1,r}
$
\\
&  $\textup{all other cases}$ & $\alpha_{p,0,r}$ & $\alpha_{p,1,r}$. \\ 
\end{tabular}
}
\setstretch{1.0}
\end{center}
The distribution of 
$\dim_{\F_p} K_{2n}(\cO_F)/p$ is the distribution of 
$\dim_{\F_p} (K_{2n}(\cO_F)/p)^-$ shifted by 
$\kappa_{2n,p}$.
\end{conjecture}

\subsection{Average order}

Proposition~\ref{P:u-values for quadratic fields}
combined with the reasoning of Section~\ref{S:heuristics for class groups} 
(Proposition~\ref{P:alpha distribution}\eqref{I:alpha average}, in particular)
suggests the following statement:

\begin{conjecture}[Average order of class group components in residue classes]
\label{C:average class group heuristic}
Fix an odd prime $p$. 
The average order of $\left( \Cl(\cO_E[1/p])/p \right)^\chi$ 
for $F$ ranging over the quadratic fields $\Q(\sqrt{d})$ 
satisfying the conditions defining a box below is as follows:
\begin{center}
\begin{tabular}{c|c||c|c}
&& $d>0$ & $d<0$ \\ \hline \hline
$\textup{$n$ even}$ & $n \equiv 0 \pmod{p-1}$, $d \in \Q_p^{\times 2}$ & $1+p^{-2}$ & $1+p^{-1}$ \\ 
& $n \equiv \frac{p-1}{2} \pmod{p-1}$, $p^*d \in \Q_p^{\times 2}$ & $1+p^{-2}$ & $1+p^{-1}$ \\ 
& $\textup{all other cases}$ & $1+p^{-1}$ & $2$ \\ \hline
$\textup{$n$ odd}$ & $n \equiv \frac{p-1}{2} \pmod{p-1}$, $p^* d \in \Q_p^{\times 2}$ & $1+p^{-1}$ & $1+p^{-2}$ \\ 
&  $\textup{all other cases}$ & $2$ & $1+p^{-1}$ \\ 
\end{tabular}
\end{center}
\end{conjecture}

Conjecture~\ref{C:average class group heuristic}, in turn,
would imply the following:

\begin{conjecture}[Average order of $K$-groups modulo $p$ in residue classes]
\label{C:average K-theory in residue classes conjecture}
Fix $n \ge 1$ and an odd prime $p$.
The average order of $(K_{2n}(\cO_F)/p)^-$
for $F$ ranging over the quadratic fields $\Q(\sqrt{d})$ 
satisfying the conditions defining a box below is as follows:
\begin{center}
\begin{tabular}{c|c||c|c}
&& $d>0$ & $d<0$ \\ \hline \hline
$\textup{$n$ even}$ & $n \equiv 0 \pmod{p-1}$, $d \in \Q_p^{\times 2}$ & $p+p^{-1}$ & $p+1$ \\ 
& $n \equiv \frac{p-1}{2} \pmod{p-1}$, $p^*d \in \Q_p^{\times 2}$ & $p+p^{-1}$ & $p+1$ \\ 
& $\textup{all other cases}$ & $1+p^{-1}$ & $2$ \\ \hline
$\textup{$n$ odd}$ & $n \equiv \frac{p-1}{2} \pmod{p-1}$, $p^* d \in \Q_p^{\times 2}$ & $p+1$ & $p+p^{-1}$ \\ 
&  $\textup{all other cases}$ & $2$ & $1+p^{-1}$ \\ 
\end{tabular}
\end{center}
To get the average order of $K_{2n}(\cO_F)/p$ itself,
multiply each entry by $p^{\kappa_{2n,p}}$.
\end{conjecture}

Conjecture~\ref{C:average K-theory in residue classes conjecture}
implies Conjecture~\ref{C:average K-theory conjecture} below
for the family of all quadratic fields of given signature, 
by taking weighted combinations of averages.
For example, for $n \equiv 0 \pmod{p-1}$ and real quadratic fields
(cf.\ Example~\ref{Ex:upper left}),
\begin{align*}
	\Average \#(K_{2n}(\cO_F)/p)^-
	&= \Prob(d \in \Q_p^{\times 2}) \; (p+p^{-1}) 
	+ \Prob(d \notin \Q_p^{\times 2}) \; (1+p^{-1}) \\
	&= \frac{p}{2p+2} (p+p^{-1}) + \frac{p+2}{2p+2} (1+p^{-1}) \\
	&= \frac{p^3+p^2+4p+2}{2p^2+2p}.
\end{align*}

\begin{conjecture}[Average order of $K$-groups modulo $p$]
\label{C:average K-theory conjecture}
Fix $n \ge 1$ and an odd prime $p$.
The average order of $(K_{2n}(\cO_F)/p)^-$
for $F$ ranging over all 
real \textup{(}resp.\ imaginary\textup{)} quadratic fields
is given in the following table
by the entry in the row determined by $n$ 
and column determined by the signature:
\begin{center}
\setstretch{1.5}
\begin{tabular}{c|c||c|c}
&& \textup{real} & \textup{imaginary} \\ \hline \hline
$\textup{$n$ even}$ & $n \equiv 0 \pmod{p-1}$ & $\frac{p^3+p^2+4p+2}{2p^2+2p}$ & $\frac{p^2+3p+4}{2p+2}$ \\ 
& $n \equiv \frac{p-1}{2} \pmod{p-1}$ & $\frac{3p^2+3p+2}{2p^2+2p}$ & $\frac{5p+3}{2p+2}$ \\ 
& $\textup{all other cases}$ & $\frac{p+1}{p}$ & $2$ \\ \hline
$\textup{$n$ odd}$ & $n \equiv \frac{p-1}{2} \pmod{p-1}$ & $\frac{5p+3}{2p+2}$ & $\frac{3p^2+3p+2}{2p^2+2p}$ \\ 
&  $\textup{all other cases}$ & $2$ & $\frac{p+1}{p}$ \\ 
\end{tabular}
\setstretch{1.0}
\end{center}
To get the distribution of the order of $K_{2n}(\cO_F)/p$ itself,
multiply each entry by $p^{\kappa_{2n,p}}$.
\end{conjecture}

\begin{remark}
It is not quite clear that 
Conjectures \ref{C:class group heuristic}, \ref{C:K-theory conjecture},
and~\ref{C:K-theory conjecture for rank again}, 
imply Conjectures \ref{C:average class group heuristic}, 
\ref{C:average K-theory in residue classes conjecture}, 
and~\ref{C:average K-theory conjecture}, respectively.
For such implications, one would need to know
that the contribution to each average from the rare cases
of very large class groups or $K$-groups is negligible.
\end{remark}

\section{The average order of even \texorpdfstring{$K$}{K}-groups modulo 3}

In this section, we prove 
Conjecture~\ref{C:average class group heuristic} for $p=3$;
then Conjectures \ref{C:average K-theory in residue classes conjecture}
and~\ref{C:average K-theory conjecture} for $p=3$ follow too;
the latter becomes Theorem~\ref{T:average of K_3}.
For $p=3$, the table in Conjecture~\ref{C:average class group heuristic} 
to be verified simplifies to
\begin{center}
\begin{tabular}{c|c||c|c}
&& $d>0$ & $d<0$ \\ \hline \hline
$\textup{$n$ even}$ & $d \in \Q_3^{\times 2}$ & $10/9$ & $4/3$ \\ 
& $d \notin \Q_3^{\times 2}$ & $4/3$ & $2$ \\ \hline
$\textup{$n$ odd}$ & $-3d \in \Q_3^{\times 2}$ & $4/3$ & $10/9$ \\ 
&  $-3d \notin \Q_3^{\times 2}$ & $2$ & $4/3$\lefteqn{.} \\ 
\end{tabular}
\end{center}

We have $E=F(\zeta_3)$ 
and $G = \Gal(E/\Q) \isom (\Z/3\Z)^\times \times \{\pm 1\}$.
Let $H \colonequals \ker \chi$, which is generated by $(-1,(-1)^n)$.
Let $K \colonequals E^H$, which is $F$ if $n$ is even,
and $F'=\Q(\sqrt{-3d})$ if $n$ is odd.
For any $\F_3$-representation $V$ of $G$,
we have $V^H = V^G \directsum V^\chi$.
Apply this to $V = \Cl(\cO_E[1/3])/3$,
and use Lemma~\ref{L:Hochschild-Serre} twice to obtain
\[
	\Cl(\cO_K[1/3])/3 
	\isom \Cl(\Z[1/3])/3 \directsum \left( \Cl(\cO_E[1/3])/3 \right)^\chi
	\isom \left( \Cl(\cO_E[1/3])/3 \right)^\chi,
\]
since $\Cl(\Z[1/3])$ is a quotient of $\Cl(\Z)=0$.
It remains to show that for each box, 
the average order of $\Cl(\cO_K[1/3])/3$ 
is as given in the table above.

For fixed $n$,
as $d$ ranges over fundamental discriminants of fixed sign
in a fixed coset of $\Q_3^{\times 2}$ up to some bound,
the field $K$ ranges over quadratic fields 
with fundamental discriminant of fixed sign
in a fixed coset of $\Q_3^{\times 2}$ up to some bound
(the same sign and coset if $n$ is even, 
or sign and coset multiplied by $-3$ if $n$ is odd).
Thus it suffices to prove that the average order of $\Cl(\cO_K[1/3])/3$ 
for such $K$ is given by the table
\begin{center}
\begin{tabular}{c||cc}
& $d_K>0$ & $d_K<0$ \\ \hline \hline
$d_K \in \Q_3^{\times 2}$ & $10/9$ & $4/3$ \\ 
$d_K \notin \Q_3^{\times 2}$ & $4/3$ & $2$\lefteqn{.} \\
\end{tabular}
\end{center}

\begin{remark}
\label{R:cosets}
There are four cosets of $\Q_3^{\times 2}$ in $\Q_3^\times$.
Which coset contains a given fundamental discriminant $d_K$
is determined by whether $d_K$ is $1 \bmod 3$, $2 \bmod 3$,
$3 \bmod 9$, or $6 \bmod 9$.
In particular, $d_K \in \Q_3^{\times 2}$ 
if and only if $d_K \equiv 1 \pmod{3}$.
\end{remark}

\begin{remark}
The results in Sections \ref{S:class groups to cubic fields}
to~\ref{S:counting cubic fields} 
have been generalized by the second author 
to compute the average size of $\Cl(\cO_K[1/S])/3$ for an arbitrary
finite set of primes $S$ using essentially identical methods
\cite{Klagsbrun-preprint}. 
Similar methods have been used to compute the average size of
$\Cl(O_K[1/p])/3$ when $p$ splits in $K$ \cite{Wood2018}.
\end{remark}

\subsection{From class groups to cubic fields}
\label{S:class groups to cubic fields}

To compute the average order of $\Cl(\cO_K[1/3])/3$,
we adapt the strategy of Davenport and Heilbronn \cite{Davenport-Heilbronn1971},
which relies on results involving the following setup.

Let $L$ be a degree~$3$ extension of $\Q$ whose Galois closure $M$
has Galois group $S_3$.
Let $K$ be the quadratic resolvent of $L$,
that is, the degree~$2$ subfield $M^{A_3}$.

\begin{lemma}
\label{L:S_3 extension}
The following are equivalent:
\begin{enumerate}[\upshape (i)]
\item $M/K$ is unramified.
\item $d_L=d_K$.
\item $L/\Q$ is nowhere totally ramified.
\end{enumerate}
\end{lemma}

\begin{proof}
Let $\mathfrak{f}$ be the conductor of the abelian extension $M/K$.
By \cite{Hasse1930}*{Satz~3}, $d_L = N_{K/\Q}(\mathfrak{f}) d_K$;
this proves (i)$\leftrightarrow$(ii).
Also, $d_L$ is an integer square times $d_K$ (see \cite{Hasse1930}*{(1)});
this proves (ii)$\leftrightarrow$(iii).
Finally, \cite{Cohen2000}*{Proposition~8.4.1} yields (ii)$\leftrightarrow$(iv).
\end{proof}

Throughout this section, we consider two cubic fields to be the same
if they are abstractly isomorphic (so conjugate cubic fields are
counted only as one).

\begin{theorem}[cf.~\cite{Hasse1930}*{Satz~7}]
\label{T:Delone-Faddeev}
Fix a quadratic field $K$.
Then the following are naturally in bijection:
\begin{enumerate}[\upshape (i)]
\item The set of index~$3$ subgroups of $\Cl(\cO_K)$.
\item The set of unramified $\Z/3\Z$-extensions $M$ of $K$.
\item The set of cubic fields $L$ with $d_L=d_K$.
\end{enumerate}
\end{theorem}

\begin{proof}
Class field theory gives (i)$\leftrightarrow$(ii).
The nontrivial element of $\Gal(K/\Q)$ acts as $-1$ on $\Cl(\cO_K)$,
so each $M$ in (ii) is an $S_3$-extension of $\Q$.
The map (ii)$\to$(iii) sends $M$ to one of its cubic subfields $L$.
The map (iii)$\to$(ii) sends $L$ to its Galois closure $M$.
That these are bijections follows from 
Lemma~\ref{L:S_3 extension}(i)$\leftrightarrow$(ii).
\end{proof}

We are interested in $\Cl(\cO_K[1/3])$ instead of $\Cl(\cO_K)$,
so we need the following variant.

\begin{corollary}
\label{C:generalized Delone-Faddeev}
Fix a quadratic field $K$.
Then the following are naturally in bijection:
\begin{enumerate}[\upshape (i)]
\item The set of index~$3$ subgroups of $\Cl(\cO_K)[1/3]$.
\item The set of unramified $\Z/3\Z$-extensions $M$ of $K$
such that the primes above $3$ in $K$ split completely in $M/K$.
\item The set of cubic fields $L$ with $d_L=d_K$
such that if $d_K \in \Q_3^{\times 2}$ then $3$ splits completely in $L/\Q$.
\end{enumerate}
\end{corollary}

\begin{proof}
The group $\Cl(\cO_K)[1/3]$ is the quotient of $\Cl(\cO_K)$
by the group generated by the classes of the primes $\p | 3$ in $K$.
Thus we need to restrict the bijections in Theorem~\ref{T:Delone-Faddeev}
to the index~$3$ subgroups $H$ containing these classes.
Because the class field theory isomorphism $\Cl(\cO_K)/H \isom \Gal(M/K)$ 
sends $[\p]$ to $\Frob_\p$,
which is trivial if and only if $\p$ splits in $M/K$,
we obtain (i)$\leftrightarrow$(ii).
If $3$ is inert or ramified in $K/\Q$, then the prime above $3$
is of order dividing $2$ in $\Cl(\cO_K)$, 
so to require it to be in the index~$3$ subgroup is no condition.
If $3$ splits in $K/\Q$ (that is, $d_K \in \Q_3^{\times 2}$),
then the primes above $3$ in $K$ split in $M/K$
if and only if $3$ splits completely in $M/\Q$,
which is if and only if $3$ splits completely in $L/\Q$.
\end{proof}

\begin{corollary}
\label{C:class groups to cubic fields}
Let $K$ be a quadratic field.
\begin{enumerate}[\upshape (a)]
\item
If $d_K \in \Q_3^{\times 2}$, then
\[
	\# \Cl(\cO_K[1/3])/3 = 2 \, \#\{\textup{cubic fields $L$ with $d_L=d_K$ in which $3$ splits}\} + 1.
\]
\item
If $d_K \notin \Q_3^{\times 2}$, then
\[
	\# \Cl(\cO_K[1/3])/3 = 2 \, \#\{\textup{cubic fields $L$ with $d_L=d_K$}\} + 1.
\]
\end{enumerate}
\end{corollary}

\begin{proof}
For an elementary abelian $3$-group $V$,
\[
	\#V = 2 \, \#\{\textup{index~$3$ subgroups of $V$}\} + 1.
\]
Take $V = \Cl(\cO_K[1/3])/3$,
and apply
Corollary~\ref{C:generalized Delone-Faddeev}(i)$\leftrightarrow$(iii).
\end{proof}

\subsection{Counting quadratic fields}
\label{S:counting quadratic fields}

Fix a sign and a coset of $\Q_3^{\times 2}$ in $\Q_3^\times$;
by Remark~\ref{R:cosets}, each of the four cosets is defined by 
a congruence condition mod~$3$ or mod~$9$.
Let $\scriptD$ be the set of fundamental discriminants
having this sign and lying in this coset.
For $X>0$, define 
$\scriptD_{<X} \colonequals \{d \in \scriptD : |d| < X\}$.

To compute the average of the appropriate left hand side
in Corollary~\ref{C:class groups to cubic fields}
as $K$ ranges over quadratic fields with $d_K \in \scriptD$,
we compute the average number of cubic fields
appearing in the corresponding right hand side.
That is, we need the limit as $X \to \infty$ of
\begin{equation}
\label{E:quadratics over cubics}
	\frac{ \sum_{K: d_K \in \scriptD_{<X}} \#\{\textup{cubic fields $L$ with $d_L=d_K$ such that if $d_K \in \Q_3^{\times 2}$ then $3$ splits in $L$}\}}
	{\#\{\textup{quadratic fields $K$ such that $d_K \in \scriptD_{<X}$}\}}.
\end{equation}
We first compute the denominator.

\begin{proposition}
\label{P:counting quadratic fields}
The number of quadratic fields $K$ satisfying $|d_K|<X$
and prescribed sign and $3$-adic congruence conditions 
is $\alpha_2 X/\zeta(2) + o(X)$,
where $\alpha_2$ is given by the following table:
\begin{center}
\begin{tabular}{c||c|c}
& $d_K>0$ & $d_K<0$ \\ \hline \hline
$\textup{$d_K \equiv 1\pmod{3}$ }$ & $3/16$ & $3/16$ \\ 
$\textup{$d_K \equiv 2\pmod{3}$ }$ & $3/16$ & $3/16$ \\ 
$\textup{$d_K \equiv 3\pmod{9}$ }$ & $1/16$ & $1/16$ \\ 
$\textup{$d_K \equiv 6\pmod{9}$ }$ & $1/16$ & $1/16$\lefteqn{.} \\ 
\end{tabular}
\end{center}
\end{proposition}

\begin{proof}
Use an elementary squarefree sieve.
\end{proof}

\begin{remark}
Even though many entries in the table of 
Proposition~\ref{P:counting quadratic fields} coincide,
it is stronger to give the asymptotics for the individual field families
without merging them, and we need the stronger results.
\end{remark}

\subsection{Counting cubic fields}
\label{S:counting cubic fields}

By Lemma~\ref{L:S_3 extension}(ii)$\leftrightarrow$(iii), 
the numerator in~\eqref{E:quadratics over cubics}
equals the number of nowhere totally ramified cubic fields $L$ 
with $d_L \in \scriptD_{<X}$ 
such that if $d_L \in \Q_3^{\times 2}$ then $3$ splits completely in $L$.
To compute this number,
we follow the Davenport--Heilbronn approach,
in the form of a refinement due to
Bhargava, Shankar, and Tsimerman \cite{Bhargava-Shankar-Tsimerman2013}.

For every prime $p$, let $\widehat{\Sigma}_p$ 
be the set of maximal cubic $\Z_p$-orders
that are not totally ramified,
up to isomorphism.
For $R \in \widehat{\Sigma}_p$,
let $\Disc_p(R)$ be the power of $p$
generating the discriminant ideal of $R$.

\begin{theorem}\label{thm:BST}
For each prime $p$ $($including~$2$$)$, let $\Sigma_p \subseteq \widehat{\Sigma}_p$.
Suppose that $\Sigma_p = \widehat{\Sigma}_p$ for all $p$ outside
a finite set $\cP$.
Define
\[
	c_p \colonequals \frac{p}{p+1} 
	\sum_{R \in \Sigma_p} \frac{1}{\Disc_p(R) \; \lvert\Aut R\rvert}.
\]
\begin{enumerate}[\upshape (a)]
\item The number of nowhere totally ramified totally real cubic fields $L$ \textup{(}up to isomorphism\textup{)} such that $|d_L| < X$ and $\cO_L \tensor \Z_p \in \Sigma_p$ 
for all $p \in \cP$ is 
$\frac{1}{12\zeta(2)}\left (\prod_{p\in \cP} c_p \right) X + o(X)$.
\item The number of nowhere totally ramified complex cubic fields $L$ \textup{(}up to isomorphism\textup{)} such that $|d_L| < X$ and $\cO_L \tensor \Z_p \in \Sigma_p$ for all $p \in \cP$ is $\frac{1}{4\zeta(2)}\left (\prod_{p\in \cP} c_p \right) X + o(X)$.
\end{enumerate}
\end{theorem}

\begin{proof}
The definition of $c_p$ yields
\[
	\frac{p-1}{p}
	\sum_{R \in \Sigma_p} \frac{1}{\Disc_p(R) \; \lvert\Aut{R}\rvert} 
	= \left( 1 - \frac{1}{p^2} \right) c_p.
\]
For each $p \notin \cP$, 
enumerating $\widehat{\Sigma}_p$ explicitly shows that $c_p=1$.
Substituting this into \cite{Bhargava-Shankar-Tsimerman2013}*{Theorem~8}
yields the result.
\end{proof}

\begin{corollary}\label{cor:d_L1mod3}
The number of nowhere totally ramified cubic fields $L$ with $|d_L|<X$
satisfying prescribed sign and $3$-adic congruence conditions below
such that if $d_L \equiv 1 \pmod{3}$ then $3$ splits completely in $L$ 
is $\alpha_3 X/\zeta(2) + o(X)$,
where $\alpha_3$ is given by the following table:
\begin{center}
\begin{tabular}{c||c|c}
& $d_L>0$ & $d_L<0$ \\ \hline \hline
$\textup{$d_L \equiv 1\pmod{3}$ }$ & $1/96$ & $1/32$ \\ 
$\textup{$d_L \equiv 2\pmod{3}$ }$ & $1/32$ & $3/32$ \\ 
$\textup{$d_L \equiv 3\pmod{9}$ }$ & $1/96$ & $1/32$ \\ 
$\textup{$d_L \equiv 6\pmod{9}$ }$ & $1/96$ & $1/32$\lefteqn{.} \\ 
\end{tabular}
\end{center}
\end{corollary}

\begin{proof}
We apply Theorem~\ref{thm:BST} with $\cP = \{ 3\}$ 
and with $\Sigma_3$ tailored to the row.
For the first row, let $\Sigma_3 \colonequals \{\Z_3 \times \Z_3 \times \Z_3\}$, so
\[
	c_3 = \frac{3}{4} \cdot \frac{1}{\Disc_3(\Z_3 \times \Z_3 \times \Z_3) \; \lvert \Aut\left(\Z_3 \times \Z_3 \times \Z_3 \right) \rvert} = \frac{3}{4} \cdot \frac{1}{6} = \frac{1}{8}.
\]
For each other row, 
let $\Sigma_3 \colonequals \{\Z_3 \times \Z_3(\sqrt{d_L})\}$,
and calculate $c_3$ similarly.
\end{proof}

\subsection{End of proof}

Corollary~\ref{cor:d_L1mod3} and Proposition~\ref{P:counting quadratic fields}
give the asymptotic behavior of the numerator and denominator, respectively,
in \eqref{E:quadratics over cubics}
(see the first sentence of Section~\ref{S:counting cubic fields}).
Thus, as $X \to \infty$, the ratio~\eqref{E:quadratics over cubics}
tends to $\alpha_3/\alpha_2$.
Following Corollary~\ref{C:class groups to cubic fields},
we multiply by $2$ and add $1$ to obtain the average order of
$\Cl(\cO_K[1/3])/3$ as $K$ varies over quadratic fields with $d_K \in \scriptD$.
For each signature, the answer is the same for each
of the three nontrivial cosets of $\Q_3^{\times 2}$,
so we combine them into a single entry 
in the table before Remark~\ref{R:cosets}.
This completes the proof of 
Conjecture~\ref{C:average class group heuristic} for $p=3$,
and hence also 
Conjectures \ref{C:average K-theory in residue classes conjecture}
and~\ref{C:average K-theory conjecture} for $p=3$
and Theorem~\ref{T:average of K_3}.

\section*{Acknowledgments} 

We thank the referee for taking the time to provide many helpful comments.
We thank David DeMark and Craig Westerland for pointing out an error in our original proof of Lemma~\ref{L:Hochschild-Serre}, which we have now corrected.

\end{document}